\documentclass{amsart}
\usepackage[latin9]{inputenc}
\usepackage[letterpaper]{geometry}
\usepackage{color}
\usepackage{amsthm}
\usepackage{amstext}
\usepackage{amssymb}
\usepackage{esint}
\usepackage[unicode=true,
 bookmarks=false,
 breaklinks=false,pdfborder={0 0 1},backref=false,colorlinks=true]
 {hyperref}

\makeatletter
\numberwithin{equation}{section}
\numberwithin{figure}{section}
\usepackage{enumitem}		% customizable list environments
\theoremstyle{plain}
\newtheorem{thm}{\protect\theoremname}
  \theoremstyle{definition}
  \newtheorem{defn}[thm]{\protect\definitionname}
  \theoremstyle{plain}
  \newtheorem{prop}[thm]{\protect\propositionname}
  \theoremstyle{plain}
  \newtheorem{cor}[thm]{\protect\corollaryname}
  \theoremstyle{plain}
  \newtheorem{lem}[thm]{\protect\lemmaname}
  \theoremstyle{remark}
  \newtheorem{rem}[thm]{\protect\remarkname}
  \theoremstyle{definition}
  \newtheorem{example}[thm]{\protect\examplename}

\@ifundefined{date}{}{\date{}}
\usepackage{amscd}
\usepackage[all]{xy}
\usepackage{url}

\newcommand{\blank}{-}

\newcommand{\ra}{\to}

\newcommand{\N}{\ensuremath{\mathbb{N}}}
\newcommand{\Z}{\ensuremath{\mathbb{Z}}}

\newcommand{\R}{\ensuremath{\mathbb{R}}}

\newcommand{\sgn}{\text{sgn}}
\newcommand{\im}{\text{Im}}
\newcommand{\supp}{\text{Supp}}

\newcommand{\FU}{{^{\bullet} U}}
\newcommand{\FV}{{^{\bullet} V}}
\newcommand{\FW}{{^{\bullet} W}}

\newcommand{\FR}{{^{\bullet} \R}}
\newcommand{\Ff}{{^{\bullet} f}}

\newcommand{\Fh}{{^{\bullet} h}}

\newcommand{\Falpha}{{^{\bullet} \alpha}}
\newcommand{\Fbeta}{{^{\bullet} \beta}}
\newcommand{\stx}{{^{\circ}x}}
\newcommand{\sty}{{^{\circ}y}}
\newcommand{\stz}{{^{\circ}z}}
\newcommand{\stp}{{^{\circ}p}}
\newcommand{\sta}{{^{\circ}a}}
\newcommand{\stb}{{^{\circ}b}}
\newcommand{\stv}{{^{\circ}v}}

\makeatother

  \providecommand{\corollaryname}{Corollary}
  \providecommand{\definitionname}{Definition}
  \providecommand{\examplename}{Example}
  \providecommand{\lemmaname}{Lemma}
  \providecommand{\propositionname}{Proposition}
  \providecommand{\remarkname}{Remark}
\providecommand{\theoremname}{Theorem}

\begin{document}

\title{Convergences and the Intermediate Value Property in Fermat Reals}

\author{Enxin Wu}

\thanks{E..~Wu has been supported by grant P25311-N25 of the Austrian Science
Fund FWF}

\address{\textsc{University of Vienna, Austria}}

\email{enxin.wu@univie.ac.at}


\begin{abstract}
This paper contains two topics of Fermat reals, as suggested by the
title. In the first part, we study the $\omega$-topology, the order
topology and the Euclidean topology on Fermat reals, and their convergence
properties, with emphasis on the relationship with the convergence
of sequences of ordinary smooth functions. We show that the Euclidean
topology is best for this relationship with respect to pointwise convergence,
and Lebesgue dominated convergence does not hold, among all additive
Hausdorff topologies on Fermat reals. In the second part, we study
the intermediate value property of quasi-standard smooth functions
on Fermat reals, together with some easy applications. The paper is
written in the language of Fermat reals, and the idea could be extended
to other similar situations.
\end{abstract}

\maketitle
\tableofcontents{}

\section{Introduction}

The idea of using infinitesimals in geometry and analysis, even from
its birth, was on the one hand very intuitive and computable, and
hence led to great development of mathematics and physics, and on
the other hand very controversial for its rigor. It was A.-L. Cauchy
who made the definition of limit rigorous using the epsilon-delta
language. Since then, infinitesimals gradually left the main stream
of mathematics, but its idea was still used while doing research.
The renaissance of infinitesimals happened when they were made rigorous,
together with many applications in other fields of mathematics (see
for example Non-Standard Analysis \cite{R} and Synthetic Differential
Geometry \cite{K}). 

Among all the existing infinitesimal theories, the theory of Fermat
reals introduced by P. Giordano in~\cite{G2} has the properties
that the theory is compatible with classical logic, all infinitesimals
are nilpotent, and the ring $\FR$ of Fermat reals is well-ordered.
The whole theory is a mixture of algebra and analysis: the model of
infinitesimals are polynomial-like function (called little-oh polynomials)
modulo certain degree, the functions (called quasi-standard smooth
functions) are locally extensions of ordinary smooth functions with
parameters, and the calculations are given by Taylor's expansion at
standard point together with the nilpotency of infinitesimals (and
hence a finite sum); see Section~\ref{sec:Basics} for a quick review
of the basics of Fermat reals. 

In the current paper, we continue developing calculus of Fermat reals
(see \cite{GW} for the integral calculus). More precisely, we study
two questions: (1) Does Lebesgue dominated convergence hold in Fermat
reals? (2) Does every quasi-standard smooth function (of one variable)
has the intermediate value property? 

To settle the first question, we first study three natural topologies
on Fermat reals (the $\omega$-topology, the order topology and the
Euclidean topology) and their properties of convergence (pointwise
and uniform), with emphasis on the relationship with the convergence
of ordinary smooth functions. Then we show that the Euclidean topology
is best for pointwise convergence (Theorem~\ref{thm:best}), and
by a similar method that the Lebesgue dominated convergence does not
hold (Theorem~\ref{thm:Lebesgue}), for any additive Hausdorff topologies
on the Fermat reals.

For the second question, the general answer is no (Remark~\ref{rem:slice image}
\eqref{enu:no intermediate value}). So the real interesting question
is, which quasi-standard smooth functions have the intermediate value
property. We study this in depth from simple to general, together
with (counter-)examples and some applications. We show in Corollary~\ref{cor:intermediate for Fermat extension}
that the extension (without parameter) of ordinary smooth functions
with no flat point have the intermediate value property, and the general
case is solved in Subsection~\ref{sub:IV-for-quasi} (especially
Proposition~\ref{prop:intermediate}) by a similar method. The proof
of Corollary~\ref{cor:intermediate for Fermat extension} contains
three ingredients: the slice image theorem (Theorem~\ref{thm:slice image theorem}),
the slice monotonicity (Observation~\eqref{enu:slice-monotonicity}
in Subsection~\ref{sub:Monotonicity}) and some real analysis (the
proof of Corollary~\ref{cor:intermediate for Fermat extension}).
And the slice image theorem (Theorem~\ref{thm:slice image theorem})
is indeed an algorithm, whose finite termination is proved with the
usage of an unusual method (a mixture of real and symbolic computations). 

Although this paper is written in the language of Fermat reals, many
examples and some methods of proof can be applied to other similar
situations.

\medskip{}

I would like to thank P. Giordano for raising the question of Lebesgue
dominated convergence in the collaboration of \cite{GW} together
with some comments on the first draft of Section~\ref{sec:IVP},
and also to G. Sinnamon for providing Example~\ref{exa:complicatedegs}
and some discussion of Example~\ref{exa:flat-ivt}.

\section{\label{sec:Basics}Basics on Fermat reals}

Fermat reals were introduced by P.~Giordano in~\cite{G1,G2,G3,GK}.
Let us review the basic theory here; see these references for detailed
proof of these results.

Let $U$ be an open subset of $\R^{n}$. We define $U_{0}[t]$, the
\emph{little-oh polynomials} on $U$, to be the set of functions $x:[0,\epsilon)\ra U$
for some (not fixed) $\epsilon\in\R_{>0}$ with the property that
\[
\|x(t)-r-\sum_{i=1}^{k}\alpha_{i}t^{a_{i}}\|=o(t)\quad\text{i.e.,}\quad\lim_{t\ra0+}\frac{\|x(t)-r-\sum_{i=1}^{k}\alpha_{i}t^{a_{i}}\|}{t}=0
\]
for some $r\in U$, $k\in\N$, $\alpha_{i}\in\R^{n}$ and $a_{i}>0$.
Two little-oh polynomials $x$ and $y$ are called equivalent if $x(0)=y(0)$
and $x(t)-y(t)=o(t)$. This is an equivalence relation on $U_{0}[t]$,
and the quotient set is denoted by $\FU$. As a consequence, every
element in $\FU$ has a unique representing little-oh polynomial of
the form 
\begin{equation}
y(t)=\sty+\sum_{i=1}^{l}\beta_{i}t^{b_{i}}\label{eq:decomposition}
\end{equation}
 for some $\sty(:=y(0))\in U$, $l\in\N$, $\beta_{i}\in(\R^{n}\setminus\{0\})$
and $0<b_{1}<b_{2}<\cdots<b_{l}\leq1$, defined on $[0,\delta)$ for
some maximum $\delta\in\R_{>0}\cup\{\infty\}$. We call this the \emph{decomposition}
of the element $[y]$, $\sty$ the \emph{standard part}, and we define
$\omega([y]):=\frac{1}{b_{1}}$ the \emph{order} of $[y]$. For convenience,
we sometimes use a similar form of $y(t)$ as \eqref{eq:decomposition}
but allowing $\beta_{i}=0$, and we call such a form a \emph{quasi-decomposition}
of $[y]$. From now on, we write elements in $\FU$ by $y$ instead
of $[y]$ whenever there is no confusion.

Given a finite set of open subsets $\{U_{i}\}_{i\in I}$ of Euclidean
spaces, $^{\bullet}(\prod_{i\in I}U_{i})$ naturally bijects $\prod_{i\in I}\FU_{i}$.
Therefore, we do not distinguish $^{\bullet}(\R^{n})$ and $(\FR)^{n}$,
and write it as $\FR^{n}$. We can also identify $\FU$ as a subset
of $\FR^{n}$ by $\FU=\{x\in\FR^{n}\mid\stx\in U\}$ when $U$ is
an open subset of $\R^{n}$.

There is a canonical injective map $i_{U}:U\ra\FU$ defined by $i_{U}(u)(t)=u$.
So $\FU$ is an extension of $U$, and for $x\in\FU$, we call $\delta x:=x-\stx$
the \emph{infinitesimal part} of $x$. The meaning is clear when $U=\R$:
we can give a well ordering on $\FR$%
\footnote{ It is a commutative unital ring under pointwise addition and pointwise
multiplication, called the ring of \emph{Fermat reals.}%
} by $x\leq y$ if $x=\stx+\sum_{i=1}^{n}\alpha_{i}t^{a_{i}}$ and
$y=\sty+\sum_{i=1}^{n}\beta_{i}t^{a_{i}}$, both in the quasi-standard
form, with $(\stx,\alpha_{1,}\ldots,\alpha_{n})\leq(\sty,\beta_{1},\ldots,\beta_{n})$
in the dictionary order, and then $D_{\infty}:=\{x\in\FR\mid\stx=0\}=\{x\in\FR\mid-r<x<r\text{ for all }r\in\R_{>0}\}$.
Moreover, every infinitesimal part $\delta x$ of $x\in\FU$ is nilpotent,
i.e., there exists some $m=m(x)\in\N$ such that $(\delta x)^{m}=0$.

Using this ordering, we can define intervals on $\FR$, e.g. $(0,1):=\{x\in\FR\mid0<x<1\}$.
Instead, the usual intervals on $\R$ will be denoted, e.g. $(0,1)_{\R}=(0,1)\cap\R$.

On $\FR^{n}$, define $\tau:=\{\FU\mid U\text{ is an open subset of }\R^{n}\}$.
Then $\tau$ is a topology on $\FR^{n}$, called the \emph{Fermat
topology}, since $^{\bullet}(U\cap V)=\FU\cap\FV$ and $^{\bullet}(\cup_{i}U_{i})=\cup_{i}\FU_{i}$.
Without specification, for every subset $A$ of $\FR^{n}$, we always
equip it with the sub-topology of the Fermat topology of $\FR^{n}$.

Let $f:U\ra V$ be a smooth map between open subsets of Euclidean
spaces. Then $\Ff:\FU\ra\FV$ by $\Ff(x)=f\circ x$ is a well-defined
map extending $f$ (called the \emph{Fermat extension of $f$}), i.e.,
$\Ff(u)=f(u)$ whenever $u\in U$. The calculation of $\Ff(x)=\Ff(\stx+\delta x)$
can be done by Taylor's expansion of $f$ at the point $\stx$, using
the nilpotency of $\delta x$. More precisely, if the $(m+1)^{th}$
power of each component of $\delta x$ is $0$ for some $m\in\N$,
then we have 
\[
\Ff(x)=\Ff(\stx+\delta x)=\sum_{i\in\N^{m},|i|\leq m}\frac{1}{i!}\frac{\partial^{|i|}f}{\partial x^{i}}(\stx)\cdot(\delta x)^{i}.
\]
Therefore, for any open subset $W$ of $V$, we have $^{\bullet}(f^{-1}(W))=(\Ff)^{-1}(\FW)$,
i.e., $\Ff$ is continuous with respect to the Fermat topology.

Note that when $U\neq\emptyset$ and $\dim(V)>0$, not every constant
map $\FU\ra\FV$ is of the form $\Ff$ for some smooth map $f:U\ra V$,
since otherwise $\Ff(u)\in V\subset\FV$ for every $u\in U\subseteq\FU$.
We introduce the following definition:
\begin{defn}
Let $A\subseteq\FR^{n}$ and $B\subseteq\FR^{m}$ be arbitrary subsets.
A function $f:A\ra B$ is called \emph{quasi-standard smooth} if for
every $a\in A$, there exists an open neighborhood $U$ of $^{\circ}a$
in $\R^{n}$, an open subset $V$ of some Euclidean space, a smooth
map $\alpha:V\times U\ra\R^{m}$ and some fixed point $v\in\FV$,
such that for every $x\in A\cap\FU$, we have 
\[
f(x)={}^{\bullet}\alpha(v,x).
\]

\end{defn}
In particular, every constant map $A\ra B$ and $i_{U}:U\ra\FU$ are
quasi-standard smooth. Moreover, every quasi-standard smooth map is
continuous with respect to the Fermat topology.

\section{Quasi-standard smooth functions revisited}

In this section, we give another characterization of quasi-standard
smooth functions.
\begin{prop}
\label{prop:qscoo} Let $A$ be a subset of $\FR^{n}$. Then $f:A\ra\FR$
is a quasi-standard smooth function if and only if for every $a\in A$,
there exist an open neighborhood $U$ of $^{\circ}a$ in $\R^{n}$,
some $m=m(f,U)\in\N$, a finite number of ordinary smooth functions
$\{\alpha_{i}:U\ra\R\}_{i=0}^{m}$, and $a_{1},\ldots,a_{m}\in\R$
with $0<a_{1}<a_{2}<\ldots<a_{m}\leq1$ such that 
\begin{equation}
f(x)=\Falpha_{0}(x)+\sum_{i=1}^{m}\Falpha_{i}(x)\cdot t^{a_{i}}~~\forall x\in A\cap\FU.\label{eq:qscoo}
\end{equation}
\end{prop}
\begin{proof}
($\Rightarrow$) Let $\Falpha(p,\blank):A\cap\FU\ra\FR$ be a local
expression of $f$ near $a\in A$, where $p$ is a fixed parameter.
Then the result follows from rearranging the terms according to the
decomposition of $\delta p,(\delta p)^{2},\ldots,(\delta p)^{k}$
for $k=\omega(\delta p)$, after Taylor's expansion of $\Falpha(p,\blank)=\Falpha(\stp+\delta p,\blank)$
with respect to $\stp$. Here we have also used the fact that total
Taylor's expansion of a smooth function with several variables (for
nilpotent infinitesimals) is the same as Taylor's expansion by one
variable after another.

($\Leftarrow$) This is clear.
\end{proof}
The key point of the above proposition is, if we further assume that
$^{\circ}A:=\{\stx\mid x\in A\}$ is an open subset of $\R^{n}$ and
$^{\circ}A\subseteq A$, then the expression in (\ref{eq:qscoo})
is unique, while the expression in~\cite[Theorem~12.1.9]{G1} is
not in general. Here is the proof. Assume that we have two expressions:
\[
\begin{split}f(x) & =\Falpha_{0}(x)+\sum_{i=1}^{m}\Falpha_{i}(x)\cdot t^{a_{i}}\\
 & =\Fbeta_{0}(x)+\sum_{j=1}^{l}\Fbeta_{j}(x)\cdot t^{b_{j}}
\end{split}
\]
for all $x\in A\cap\FU$ with $m,l\in\N$, $\alpha_{i}:U\ra\R$ and
$\beta_{j}:U\ra\R$ ordinary smooth functions, and $0<a_{1}<a_{2}<\ldots<a_{m}\leq1$,
$0<b_{1}<b_{2}<\ldots<b_{l}\leq1$. We may assume that $U\subseteq{^{\circ}A}$
since $^{\circ}A$ is open in $\R^{n}$. For every $x\in{^{\circ}A}\subseteq A$,
we can conclude that $m=l$, $\{a_{1},\ldots,a_{m}\}=\{b_{1},\ldots,b_{m}\}$,
and $\alpha_{i}(x)=\beta_{i}(x)$ for $i=0,1,\ldots,m$ by the uniqueness
of decomposition of elements in $\FR$. Hence, $\Falpha_{i}(x)=\Fbeta_{i}(x)$
for each $i$.
\begin{cor}
Let $A$ be a subset of $\FR^{n}$ such that $^{\circ}A$ is an open
subset of $\R^{n}$ and $^{\circ}A\subseteq A$. Then $f:A\ra\FR$
is a quasi-standard smooth function if and only if for every precompact
subset $K$ of $A$ in the Fermat topology (i.e., the closure of $^{\circ}K$
is compact in $\R^{n}$), there exist $m=m(f,K)\in\N$, a finite number
of ordinary smooth functions $\{\alpha_{i}:U\ra\R\}_{i=0}^{m}$ with
$U$ an open neighborhood of $^{\circ}K$ in $\R^{n}$, and $a_{1},\ldots,a_{m}\in\R$
with $0<a_{1}<a_{2}<\ldots<a_{m}\leq1$ such that 
\[
f(x)=\Falpha_{0}(x)+\sum_{i=1}^{m}\Falpha_{i}(x)\cdot t^{a_{i}}~~\forall x\in K.
\]
\end{cor}
\begin{proof}
This is straightforward from the above discussion together with Proposition~\ref{prop:qscoo}.
\end{proof}

\section{Topologies and convergences in Fermat reals}

The main focus of the first part of this paper is to discuss convergences
in Fermat reals. To define convergences, we need a topology on Fermat
reals, and in order to make limit unique, we need the topology to
be Hausdorff. Since the Fermat topology is not Hausdorff, we will
introduce and study new Hausdorff topologies: the $\omega$-topology,
the order topology,%
\footnote{Set theoretically, $\FR=\R\times D_{\infty}$, i.e., the Cartesian
product of starndard part and infinitesimal part. The $\omega$-topology
essentially relates to the starndard part (i.e., the points in a small
neighborhood only differ from the standard part), and the order topology
essentially relates to the infinitesimal part.%
} and the Euclidean topology. Note that the $\omega$-topology was
first introduced in~\cite{GK}. We also explore the properties of
convergences with respect to these topologies, together with comparisons
to the convergences of ordinary smooth functions.

We first fix some notations:
\begin{defn}
A topology on a group is called \emph{additive} if the group operations
are continuous with respect to this topology. In other words, the
group with this topology is a topological group.

\label{de:topology} A topology on Fermat reals is an additive Hausdorff
topology on $\FR$, which then induces the product topology on $\FR^{n}$
for each $n\in\N$.%
\footnote{As a convention, from now on, whenever there is no adjective in front
of the word ``topology'' for Fermat reals, we mean the topology in
this sense; otherwise, it has the usual meaning.%
}
\end{defn}
Note that $\FR^{n}$ with coordinate-wise addition and the induced
topology is a topological group, since topological groups are closed
under finite products.
\begin{defn}
Let $\tau$ be a topology on Fermat reals. A sequence $(f_{n}:U\ra\FR)_{n\in\N}$
of quasi-standard smooth functions from $U\subseteq\FR^{k}$ is called
\emph{pointwise convergent in }$\tau$ if for each $x\in U$, $\lim_{n\ra\infty}f_{n}(x)$
exists in $\tau$. In other words, there exists a function (not necessarily
quasi-standard smooth; see Example~\ref{eg:ptwise-convergent}) $f:U\ra\FR$
with the property that for every $x\in U$, for any $\tau$-open neighborhood
$T$ of $f(x)$, there exists $N=N(x)\in\N$, such that for any $n>N$,
we have $f_{n}(x)\in T$.
\end{defn}
Note that we do not need additivity of the topology to define pointwise
convergence, but we need it for uniform convergence:
\begin{defn}
Let $\tau$ be a topology on Fermat reals. A sequence $(f_{n}:U\ra\FR)_{n\in\N}$
of quasi-standard smooth functions from $U\subseteq\FR^{k}$ is called
\emph{uniformly convergent in }$\tau$ if there exists a function
(not necessarily quasi-standard smooth) $f:U\ra\FR$ with the property
that for any $\tau$-open neighborhood $T$ of $0\in\FR$, there exists
$N\in\N$ such that for any $n>N$, we have $f_{n}(x)-f(x)\in T$
for every $x\in U$.
\end{defn}
Note that the convergence (both pointwise and uniform) of a sequence
of quasi-standard smooth functions only depends on the topology of
the codomain.

\subsection{The $\omega$-topology}

The $\omega$-topology on Fermat reals was first introduced in~\cite{GK}.
We review some basics of the $\omega$-topology here without any details.
The $\omega$-topology on $\FR^{n}$ is induced by a complete metric
$d_{\omega}:\FR^{n}\times\FR^{n}\ra\R_{\geq0}$ defined by $d_{\omega}(x,y)=\|\stx-\sty\|+\sum_{i=1}^{n}\omega(x_{i}-y_{i})$,
where $x_{i},y_{i}\in\FR$ are the $i^{th}$-coordinates of $x,y\in\FR^{n}$
respectively. It has a base consisting of all balls $B_{s}(x;d_{\omega})$
for $x\in\FR^{n}$ and $s\in(0,1]_{\R}$, where $B_{s}(x;d_{\omega})$
is simply $\{x+r\mid r\in\R^{n},\|r\|<s\}$. It is clear that the
$\omega$-topology on $\FR^{n}$ defined this way coincides with the
product topology of the $\omega$-topology on $\FR$, and the restriction
of the $\omega$-topology to $\R^{n}$ is the standard topology. Moreover,
the $\omega$-topology is strictly finer than the Fermat topology,
and it does not behave well with quasi-standard smooth functions (see~\cite[Section~3]{GW}).
In other words, $\FR$ with the $\omega$-topology is a topological
group, but not a topological $\R$-vector space.

Compared to convergences of sequences of ordinary smooth functions,
convergences in the $\omega$-topology is very restrictive; see the
following three results.
\begin{lem}
\label{lem:ptwise-convergent-omega} Let $(a_{n})_{n\in\N}$ be a
sequence in $\FR$. If it converges in the $\omega$-topology, then
there exist a convergent sequence $(b_{n})_{n\in\N}$ in $\R$ and
$N\in\N$ such that for all $n>N$, we have $a_{n}=a_{N}+b_{n}$.
\end{lem}
The converse of the above lemma is trivially true.
\begin{proof}
Write $a$ for the limit of the sequence $(a_{n})_{n\in\N}$ in the
$\omega$-topology. Since the $\omega$-open neighborhood $B_{1}(a;d_{\omega})$
of $a$ is the set $\{a+r\mid r\in\R,|r|<1\}$, we know that there
exists $N\in\N$ such that for every $n\geq N$, we have $a_{n}\in B_{1}(a;d_{\omega})$.
In other words, $a_{n}-a_{N}\in\R$ for $n\geq N$. We set 
\[
b_{n}=\begin{cases}
a_{n}-a_{N}, & \textrm{if \ensuremath{n>N}}\\
0, & \textrm{otherwise.}
\end{cases}
\]
The rest of the proof is easy.\end{proof}
\begin{prop}
Let $(f_{n}:U\ra\R)$ be a sequence of ordinary smooth functions from
an open connected subset $U\subseteq\R$. If the sequence $(\Ff_{n}:\FU\ra\FR)_{n\in\N}$
converges uniformly in the $\omega$-topology, then there exists $N\in\N$
and a convergent sequence $(a_{n})_{n\in\N}$ in $\R$ such that for
every $n>N$, we have $f_{n}=f_{N}+a_{n}$.
\end{prop}
The converse of the above proposition is trivially true.
\begin{proof}
By the previous lemma and the definition of uniform convergence, we
know that there exists $N\in\N$ such that for any $n>N$ and any
$u\in\FU$, $\Ff_{n}(u)$ and $\Ff_{N}(u)$ only differ from the standard
part. In particular, this implies that for any $n>N$ and any $x\in U$,
we have $f_{n}'(x)=f_{N}'(x)$. By the constant function theorem,
we know that $f_{n}-f_{N}$ is constant for every $n>N$, since the
domain $U$ of these functions is connected. Let $x_{0}\in U$ be
any point, and define 
\[
a_{n}=\begin{cases}
f_{n}(x_{0})-f_{N}(x_{0}), & \textrm{if \ensuremath{n>N}}\\
0, & \textrm{otherwise.}
\end{cases}
\]
The rest of the proof is easy.
\end{proof}
More generally, we have
\begin{thm}
Let $U$ be a connected open subset of $\R^{n}$, and let $A$ be
a subset of $\FR^{n}$ such that $\FU\subseteq A\subseteq\overline{\FU}$,
where $\overline{\FU}$ denotes the closure of $\FU$ in $\FR^{n}$
with respect to the $\omega$-topology. If $(f_{m}:A\ra\FR)_{m\in\N}$
is a uniformly convergent sequence of quasi-standard smooth functions
in the $\omega$-topology, then there exist a convergent sequence
$(a_{m})_{m\in\N}$ in $\R$ and $N\in\N$ such that for every $m>N$,
we have $f_{m}=f_{N}+a_{m}$. In particular, the limit function is
also quasi-standard smooth.
\end{thm}
The converse of the above theorem is trivially true.
\begin{proof}
Note that $\overline{\FU}=\FR^{n}\setminus{^{\bullet}(\R^{n}\setminus\bar{U})}$,
where $\bar{U}$ is the closure of $U$ in $\R^{n}$. By the uniqueness
theorem (\cite[Theorem~5]{GW}), it is enough to prove the statement
for $A=\FU$.

Let $f$ be the limit of the sequence $(f_{m}:A\ra\FR)_{m\in\N}$
of quasi-standard smooth functions. By the definition of uniform convergence
and Lemma~\ref{lem:ptwise-convergent-omega}, we know that there
exists $N\in\N$ such that for any $m\geq N$ and any $x\in\FU$,
$f_{m}(x)$ and $f(x)$ only differ from the standard part. Hence,
$f_{m}(x)$ and $f_{N}(x)$ only differ from the standard part. In
other words, we have a quasi-standard smooth function $g_{m}:\FU\ra\R$
defined by $x\mapsto f_{m}(x)-f_{N}(x)$. Since $U$ is connected
in $\R^{n}$, $\FU$ is connected in $\FR^{n}$ with respect to the
Fermat topology. By~\cite[Proposition~24]{W}, we know that $g_{m}$
is constant. The rest of the proof is easy.
\end{proof}

\subsection{The order topology}

Since $\FR$ is a totally ordered ring, we have the order topology
on $\FR$. Note that the order topology on $\FR$ is a Hausdorff topology,
which has a base consisting of open intervals $(a,b)=\{x\in\FR\mid a<x<b\}$
for $a,b\in\FR$. Hence, the restriction of the order topology to
$\R$ is the discrete topology (for example, $(-t,t)\cap\R=\{0\}$).
\begin{prop}
$\FR$ with the order topology is a topological group, but not a topological
$\R$-vector space. In other words, the order topology does not behave
well with quasi-standard smooth functions.\end{prop}
\begin{proof}
Note that the preimage of the order open subset $(t^{1/2}-t,t^{1/2}+t)$
of the scalar multiplication map $\R\times\FR\ra\FR$ contains $(1,t^{1/2})\in\R\times\FR$,
but none of its open neighborhoods. The rest is straightforward.\end{proof}
\begin{lem}
The order topology on $\FR^{n}$ is first countable.\end{lem}
\begin{proof}
It is enough to prove this for $n=1$. By the previous proposition,
it is enough to show that $0\in\FR$ has a countable neighborhood
base in the order topology, which can in fact be chosen as $(-\frac{1}{n}t,\frac{1}{n}t)$
for $n\in\Z^{>0}$.
\end{proof}
It is clear that the order topology on $\FR^{n}$ contains the Fermat
topology, but it is not comparable with the $\omega$-topology.

Compared to convergences of sequences of ordinary smooth functions,
convergences in the order topology is also very restrictive; see the
following two results.
\begin{lem}
\label{lem:ptwise-convergent-order} Let $(a_{n})_{n\in\N}$ be a
sequence in $\FR$. If it converges in the order topology, then there
exists $N\in\N$ such that for all $n>N$, $a_{n}-a_{N}=b_{n}t$ for
some convergent sequence $(b_{n})_{n\in\N}$ in $\R$. In other words,
the coefficient of $t^{i}$ in the decomposition of $a_{n}$ are fixed
for all $0<i<1$ and all $n>N$. 
\end{lem}
The converse of the above lemma is trivially true.
\begin{proof}
Assume that the sequence $(a_{n})_{n\in\N}$ converges to $a\in\FR$
in the order topology. The result follows directly by considering
the order open neighborhoods $(a-\frac{1}{m}t,a+\frac{1}{m}t)$ of
$a$ for $m\in\Z^{>0}$.\end{proof}
\begin{prop}
Let $(f_{n}:U\ra\R)_{n\in\N}$ be a sequence of ordinary smooth functions
from an open subset $U\subseteq\R^{m}$. If the sequence $(\Ff_{n}:\FU\ra\FR)_{n\in\N}$
converges uniformly in the order topology, then there exists $N\in\N$
such that for every $n>N$, we have $f_{n}=f_{N}$.
\end{prop}
The converse of the above proposition is trivially true.
\begin{proof}
By the previous lemma and the definition of uniform convergence, we
know that there exists $N\in\N$ such that for any $n>N$ and any
$u\in\FU$, $\Ff_{n}(u)$ and $\Ff_{N}(u)$ only differ from the coefficients
of $t$ in their decomposition. In particular, $f_{n}=f_{N}$ for
all $n>N$, by taking the standard part.
\end{proof}
As \cite[Theorem~5]{GW}, a similar uniqueness theorem holds for the
order topology:
\begin{thm}
\label{thm:uniqueness-order}Let $A$ be an arbitrary subset of $\FR^{n}$,
and let $f,g:A\ra\FR$ be quasi-standard smooth functions. If $f(x)=g(x)$
for all $x$ in a dense subset $B$ of $A$ in the order topology,
then $f=g$.\end{thm}
\begin{proof}
Without loss of generality, we may assume that $g=0$. By definition
of quasi-standard smooth function, assume that $f(x)=\Falpha(p,x)$
for $x\in A\cap\FU$, where $U$ is an open subset of $\R^{n}$, $W$
is an open subset of a Euclidean space, $\alpha:W\times U\ra\R$ is
an ordinary smooth function, and $p\in\FW$ is a fixed parameter.
By definition of the order topology on Fermat reals, for any $x_{0}\in(A\setminus B)\cap\FU$,
there exists a sequence $(a_{i})_{i\in\Z^{>0}}$ in $\R^{n}$ converging
to $0$, such that $x_{i}:=x_{0}+a_{i}t\in B$. By Taylor's formula
with nilpotent increments, one checks that 
\[
f(x_{0})=f(x_{0})-f(x_{i})=-\sum_{j=1}^{n}\frac{\partial\alpha}{\partial x_{j}}(\stp,\stx_{0})\cdot a_{ij}t
\]
 for every $i\in\Z^{>0}$, where $a_{ij}$ is the $j^{th}$-component
of $a_{i}$, which implies that $f(x_{0})=0$, since $a_{ij}\ra0$
as $i\ra\infty$. 
\end{proof}

\subsection{The Euclidean topology}

Note that as an $\R$-vector space, $\FR^{n}$ can be viewed as a
linear subspace of $\R^{n}\times(\R^{n})^{(0,1]_{\R}}$, consisting
of elements $(x,y)\in\R^{n}\times(\R^{n})^{(0,1]_{\R}}$ such that
$y_{i}=0\in\R^{n}$ for all except finitely many $i\in(0,1]_{\R}$.
So we can write elements in $\FR^{n}$ as $x=\stx+\sum_{i\in(0,1]_{\R}}\alpha_{i}\cdot t^{i}$
(called the \emph{quasi-decomposition} of $x\in\FR^{n}$), and remembering
that all $\alpha_{i}=0\in\R^{n}$ except for finitely many $i\in(0,1]_{\R}$.
The notion of quasi-decomposition here is consistent with that in
Section~\ref{sec:Basics}.
\begin{defn}
\label{de:Euclidean} Let $x=\stx+\sum_{i\in(0,1]_{\R}}\alpha_{i}\cdot t^{i}$
and $y=\sty+\sum_{i\in(0,1]_{\R}}\beta_{i}\cdot t^{i}$ be quasi-decomposition
of two elements in $\FR^{n}$. Recall that only finitely many $\alpha_{i}$
and $\beta_{i}$'s are non-zero. We define 
\[
\langle x,y\rangle=\langle\stx,\sty\rangle_{\R}+\sum_{i\in(0,1]_{\R}}\langle\alpha_{i},\beta_{i}\rangle_{\R},
\]
where $\langle\blank,\blank\rangle_{\R}$ denotes the standard inner
product on $\R^{n}$.\end{defn}
\begin{lem}
$\langle\blank,\blank\rangle:\FR^{n}\times\FR^{n}\ra\R$ defined above
is an inner product on the infinite-dimensional $\R$-vector space
$\FR^{n}$. It makes $\FR^{n}$ into a topological $\R$-vector space,
but not a Hilbert space.
\end{lem}
We call this inner product the \emph{Euclidean inner product} on $\FR^{n}$,
and the topology induced by it the \emph{Euclidean topology}. It is
clear that the Euclidean topology is strictly finer than the Fermat
topology, but strictly coarser than both the $\omega$- and the order
topologies. Moreover, the restriction of the Euclidean inner product
and the Euclidean topology to $\R^{n}$ is the standard inner product
and the standard topology.
\begin{proof}
It is straightforward to show that $\langle\blank,\blank\rangle$
is an inner product on $\FR^{n}$, and $\FR^{n}$ is a topological
$\R$-vector space in the Euclidean topology. To see that the Euclidean
topology is not complete, let $n=1$ and take $a_{m}=\sum_{i=1}^{m}\frac{1}{i!}t^{1/i}\in\FR$.
It is easy to see that the sequence $(a_{m})_{m\in\N}$ is Cauchy,
but it has no limit in $\FR$ with respect to the Euclidean topology.
\end{proof}
The Euclidean topology also does not behave well with quasi-standard
smooth functions:
\begin{rem}
Let $\|\blank\|$ be the norm induced by the Euclidean inner product
$\langle\blank,\blank\rangle$ on $\FR$. In general, there is no
definite inequality relating $\|xy\|$ and $\|x\|\cdot\|y\|$ for
$x,y\in\FR$. For example, 
\begin{enumerate}[leftmargin=*,label=(\roman*),align=left ]
\item $\|t\cdot t\|=0<1=\|t\|\cdot\|t\|$; 
\item $\|(1+t^{1/2})\cdot t^{1/2}\|=\sqrt{2}=\|1+t^{1/2}\|\cdot\|t^{1/2}\|$; 
\item $\|(1+t^{1/2})^{2}\|=\sqrt{6}>2=\|1+t^{1/2}\|^{2}$. 
\end{enumerate}
As a consequence, one can show that the multiplication map $\FR\times\FR\ra\FR$
is not continuous when $\FR$ is equipped with the Euclidean topology.
This is because for any $\delta\in\R_{>0}$, there exist $x_{n},y_{n}\in B_{\delta}(0;\langle\blank,\blank\rangle)$
such that $\lim_{n\ra\infty}\|x_{n}\cdot y_{n}\|=\infty$. In fact,
one can take $x_{n}=y_{n}=\frac{\delta}{\sqrt{n+1}}\sum_{i=1}^{n}t^{1/2^{i}}$.\end{rem}
\begin{thm}
The Cauchy completion of $\FR$ with respect to the Euclidean inner
product is the linear subspace $V$ of $\R\times\R^{(0,1]_{\R}}$
consisting of elements of the form $x=\stx+\sum_{i\in(0,1]_{\R}}\alpha_{i}\cdot t^{i}$
with all $\alpha_{i}=0$ except for countably many $i\in(0,1]_{\R}$,
and $\sum_{i\in(0,1]_{\R}}\alpha_{i}^{2}<\infty$. The Euclidean inner
product on $\FR$ extends canonically to an inner product on $V$,
which makes $V$ a (non-separable) Hilbert space.\end{thm}
\begin{proof}
This is straightforward.\end{proof}
\begin{rem}
We no longer have the nilpotency of infinitesimals and the dictionary
order in $V$.
\end{rem}
In order to state the next theorem, we need the following definition:
\begin{defn}
A sequence $(f_{n}:U\ra\R)_{n\in\N}$ of ordinary smooth functions
defined on an open subset $U$ of $\R^{m}$ is called \emph{pointwise
Taylor convergent}, if for every $k\in\N^{m}$, the sequence $(\frac{\partial^{|k|}f_{n}}{\partial x^{k}})_{n\in\N}$
is pointwise convergent.\end{defn}
\begin{example}
\label{ex:ptwise-Taylor-convergence} Let $f_{n}:(-1,1)_{\R}\ra\R$
be defined by $f_{n}(x)=x^{n}$. Then for any $k\in\N$, we have 
\[
f_{n}^{(k)}(x)=\begin{cases}
0, & \textrm{if \ensuremath{n<k}}\\
\frac{n!}{(n-k)!}x^{n-k}, & \textrm{otherwise.}
\end{cases}
\]
So for any $x\in(-1,1)_{\R}$ and any $k\in\N$, $\lim_{n\ra\infty}f_{n}^{(k)}(x)=0$,
and hence the sequence $(f_{n})_{n\in\N}$ is pointwise Taylor convergent.

Note that the ordinary smooth function $f_{n}$ above can be defined
on the whole $\R$, and the sequence $(f_{n})_{n\in\N}$ is pointwise
convergent on $(-1,1]_{\R}$, but the sequence of the Fermat extension
$(\Ff_{n})_{n\in\N}$ is not pointwise convergent in the Euclidean
topology at $1+\epsilon$ for any non-zero infinitesimal $\epsilon$.
\end{example}
Here is an example of a pointwise convergent sequence which is not
pointwise Taylor convergent:
\begin{example}
Let $f_{n}:\R\ra\R$ be defined by $f_{n}(x)=\frac{1}{n}\sin(nx)$.
Then $f_{n}'(x)=\cos(nx)$. So for any $x\in\R$, $\lim_{n\ra\infty}f_{n}(x)=0$,
but $\lim_{n\ra\infty}f_{n}'(x)$ does not exist in general. Hence,
the sequence $(f_{n})_{n\in\N}$ is pointwise convergent, but not
pointwise Taylor convergent.\end{example}
\begin{thm}
\label{thm:ptwise-convergence-for-Euclidean} Let $(f_{n}:U\ra\R)_{n\in\N}$
be a sequence of ordinary smooth functions defined on an open subset
$U$ of $\R^{m}$. Then $(f_{n})_{n\in\N}$ is pointwise Taylor convergent
if and only if the sequence of the Fermat extension $(\Ff_{n}:\FU\ra\FR)_{n\in\N}$
is pointwise convergent in the Euclidean topology.
\end{thm}
A similar statement does not hold if we change the Euclidean topology
to the $\omega$-topology or the order topology; see Example~\ref{ex:ptwise-Taylor-convergence},
Lemmas~\ref{lem:ptwise-convergent-omega} and~\ref{lem:ptwise-convergent-order}.
\begin{proof}
For any $x\in\FU$, write $x=\stx+\delta x$ for the standard and
the infinitesimal parts of $x$, and write $k$ for the integer part
of $\omega(x)$. Then we have 
\[
\Ff_{n}(x)=\sum_{i\in\N^{m},|i|\leq k}\frac{1}{i!}\frac{\partial^{|i|}f}{\partial x^{i}}(\stx)\cdot(\delta x)^{i}.
\]

($\Leftarrow$) For $y\in\FU$, write $\Ff_{n}^{w}(y)$ for the coefficient
of $t^{w}$ of the quasi-decomposition of $\Ff_{n}(y)$. By the definition
of the Euclidean topology, we know that the fact that $\lim_{n\ra\infty}\Ff_{n}(y)$
exists implies that $\lim_{n\ra\infty}{^{\circ}(\Ff_{n}(y))}$ and
$\lim_{n\ra\infty}\Ff_{n}^{w}(y)$ exist for each $w\in(0,1]_{\R}$.%
\footnote{As a warning, the converse of this statement is not necessarily true,
unless the sequence $(\Ff_{n}(y))_{n\in\N}$ is uniformly bounded,
in the sense that there exists a finite subset $A$ of $(0,1]_{\R}$
such that $\Ff_{n}^{w}(y)=0$ for every $w\in(0,1]_{\R}\setminus A$
and for every large enough $n$.%
} The conclusion then follows from the facts that for any $x\in U$,
$\lim_{n\ra\infty}f_{n}(x)=\lim_{n\ra\infty}\Ff_{n}(x)$, and that
$\lim_{n\ra\infty}\frac{\partial^{|j|}f_{n}}{\partial x^{j}}(x)$
is $j!$ times the coefficient of $t$ of the quasi-decomposition
of $\lim_{n\ra\infty}\Ff_{n}(x+\sum_{l=1}^{m}t^{a_{l}}\cdot e_{l})$
for each $j\in\N_{\geq0}^{m}\setminus\{0\}$, where $e_{1},\ldots,e_{m}$
is the standard basis for $\R^{m}$, and $(a_{1},\ldots,a_{m})=(a_{1}(j),\ldots,a_{m}(j))\in\R_{>0}^{m}$
are suitably chosen such that $\sum_{l=1}^{m}j_{l}a_{l}=1$, and for
any $s\in\N^{m}$ with $s\neq j$, we have $\sum_{l=1}^{m}s_{l}a_{l}\neq1$.
The existence of such $a_{l}$'s is guaranteed by the following remark.

($\Rightarrow$) Since the sequence $(f_{n})_{n\in\N}$ of ordinary
smooth functions is pointwise Taylor convergent, $\lim_{n\ra\infty}\frac{\partial^{|i|}f_{n}}{\partial x^{i}}(\stx)$
exists for each $i\in\N^{m}$, and hence $\lim_{n\ra\infty}\Ff_{n}(x)$
exists.\end{proof}
\begin{rem}
In this remark, we show that for any fixed $j\in\N^{m}\setminus\{0\}$,
there exists $a\in\R_{>0}^{m}$ such that for any $s\in\N^{m}$ with
$s\neq j$, we have $\sum_{l=1}^{m}s_{l}a_{l}\neq\sum_{l=1}^{m}j_{l}a_{l}$.
We prove this by induction on $m$. It is clearly true for $m=1$.
Assume that we have proved the statement for $m=k$, and now we consider
$m=k+1$. Note that we have a continuous function $\phi:\R_{>0}^{k+1}\times\R_{\geq0}^{k+1}\ra\R$
defined by $(b_{1},\ldots,b_{k+1},u_{1},\ldots,u_{k+1})\mapsto\sum_{l=1}^{k+1}u_{l}b_{l}$,
and $\phi({(1,\ldots,1)\times\N^{k+1}})\subseteq\N$. Our strategy
is to perturb $(1,\ldots,1)$ a bit to get the required $a$. So we
are left to show that there exists $c\in\R_{\geq0}^{k+1}$ such that
$\sum_{l=1}^{k+1}c_{l}s_{l}\neq\sum_{l=1}^{k+1}c_{l}j_{l}$ for every
$s\in\N_{\geq0}^{k+1}$ with $|s|=|j|$ and $s\neq j$. (Then $a$
defined by $a_{l}=1+\frac{c_{l}}{N}$ for some large enough $N\in\N$
is what we are looking for, by the continuity of the map $\phi$ and
the discreteness of $\N$ in $\R$.) Note that $|s|=|j|$ and $s\neq j$
imply that $(s_{1},\ldots,s_{k})\neq(j_{1},\ldots,j_{k})$. Now we
split into two cases: $(j_{1},\ldots,j_{k})\in\N^{k}\setminus\{0\}$
and $(j_{1},\ldots,j_{k})=(0,\ldots,0)$. For the first case, we get
the conclusion by the induction hypothesis and setting $c_{k+1}=0$.
For the second case, we get the conclusion by setting $(c_{1},\ldots,c_{k},c_{k+1})=(0,\ldots,0,1)$.
\end{rem}
Uniform convergence in the Euclidean topology for sequences of the
Fermat extension of ordinary smooth functions is also trivial:
\begin{prop}
Let $(f_{n}:U\ra\R)_{n\in\N}$ be a sequence of ordinary smooth functions
defined on a connected open subset $U$ of $\R^{m}$. If the sequence
of the Fermat extension $(\Ff_{n})_{n\in\N}$ converges uniformly
in the Euclidean topology, then there exist a convergent sequence
$(a_{n})_{n\in\N}$ in $\R$ and $N\in\N$ such that for every $n>N$,
we have $f_{n}=f_{N}+a_{n}$.
\end{prop}
The converse of the above proposition is trivially true.
\begin{proof}
Since $(\Ff_{n})_{n\in\N}$ converges uniformly in the Euclidean topology,
for every open neighborhood $T$ of $0\in\FR$, there exists $N\in\N$
such that for every $n,l>N$, we have $\Ff_{n}(x)-\Ff_{l}(x)\in T$
for every $x\in\FU$. This could not hold if $\frac{\partial(f_{n}-f_{l})}{\partial x_{j}}(x_{0})\neq0$
for some $j=1,2,\ldots,m$ and some $x_{0}\in U$. The rest of the
proof is straightforward.
\end{proof}

\section{Some general results for Fermat reals}

In this section, we prove some results which hold for any topology
on Fermat reals. And recall from Definition~\ref{de:topology} that
by a topology on Fermat reals, we always mean an additive Hausdorff
topology on $\FR$, which then induces the product topology on $\FR^{n}$.

\subsection{The Euclidean topology is best for pointwise convergence}

From the previous section, we know that the pointwise convergence
of the sequence of the Fermat extension of ordinary smooth functions
in any natural topology on Fermat reals always imposes extra conditions
on the sequence of the original ordinary smooth functions. In fact,
this is a common phenomenon:
\begin{thm}
There is no topology $\tau$ (Definition~\ref{de:topology}) on Fermat
reals such that for every pointwise convergent sequence of ordinary
smooth functions $(f_{n}:U\ra\R)_{n\in\Z^{>0}}$, where $U$ is an
open subset of $\R^{m}$, the sequence of the Fermat extension $(\Ff_{n})_{n\in\Z^{>0}}$
is pointwise convergent in $\tau$.\end{thm}
\begin{proof}
We prove below that every additive topology $\tau$ satisfying the
above condition for $U=\R$ is not Hausdorff.

Let $f_{n}(x)=\frac{1}{n}\sin(nx)$. Then $(f_{n})_{n\in\Z^{>0}}$
pointwise converges to the constant function with value $0$. By assumption,
the sequence $(\Ff_{n})_{n\in\Z^{>0}}$ pointwise converges in $\tau$.
So, for any $x\in D:=\{x\in\FR\mid x^{2}=0\}$, we have 
\[
\Ff_{n}(\frac{\pi}{2}+x)=\frac{1}{n}\sin(\frac{n\pi}{2})+\cos(\frac{n\pi}{2})\cdot x.
\]
Therefore, $2x\in T$ for any $\tau$-open neighborhood $T$ of $0\in\FR$.
In other words, every $\tau$-open neighborhood of $0\in\FR$ contains
$D$, which is an ideal of the commutative unital ring $\FR$ (\cite[Theorem~23]{GK}).
Therefore, $\tau$ cannot be Hausdorff.
\end{proof}
Moreover, we have:
\begin{thm}
\label{thm:best}The Euclidean topology is a best topology (Definition~\ref{de:topology})
for pointwise convergence in the sense of Theorem~\ref{thm:ptwise-convergence-for-Euclidean}.\end{thm}
\begin{proof}
We prove below that for any $k\in\N$, if an additive topology $\tau$
on Fermat reals has the property that for every sequence of ordinary
smooth functions $(f_{n}:\R\ra\R)_{n\in\Z^{>0}}$ such that $(f_{n}^{(i)})_{n\in\Z^{>0}}$
is pointwise convergent for each $i=0,1,\ldots,k$, the sequence of
the Fermat extension $(\Ff_{n})_{n\in\Z^{>0}}$ is pointwise convergent
in $\tau$, then $\tau$ cannot be Hausdorff.

It is easy to see that for any $c\in\R$, there is an ordinary smooth
function $g:\R\ra\R$ with the properties that $g(x+2)=g(x)$ and
$g(x)=g(1-x)$ for each $x\in\R$, $g(0)=g'(0)=\cdots=g^{(k)}(0)=0$
and $g^{(k+1)}(0)=c$. From these properties, one derives that 
\[
g^{(i)}(m)=\begin{cases}
g^{(i)}(0), & \textrm{if \text{m }is even}\\
(-1)^{i}g^{(i)}(0), & \textrm{if \text{m }is odd}
\end{cases}
\]
for any $i\in\N$ and any $m\in\Z$. Now we define $f_{n}(x)=\frac{1}{n^{k+1}}g(nx)$.
Then $f_{n}^{(i)}(x)=\frac{1}{n^{k+1-i}}g^{(i)}(nx)$ for any $i\in\N$.
So for any $x\in\R$, $\lim_{n\ra\infty}f_{n}^{(i)}(x)=0$ for each
$i=0,1,\ldots,k$, and $f_{n}^{(k+1)}(x)=g^{(k+1)}(nx)$.

For any $a\in D_{k+1}:=\{x\in\FR\mid x^{k+2}=0\}$, we have 
\[
\begin{split}\Ff_{n}(1+a) & =f_{n}(1)+f_{n}'(1)\cdot a+\frac{f_{n}''(1)}{2!}\cdot a^{2}+\cdots+\frac{f_{n}^{(k)}(1)}{k!}\cdot a^{k}+\frac{f_{n}^{(k+1)}(1)}{(k+1)!}\cdot a^{k+1}\\
 & =\frac{g^{(k+1)}(n)}{(k+1)!}\cdot a^{k+1}.
\end{split}
\]
By assumption, the sequence $(\Ff_{n}(1+a))_{n\in\Z^{>0}}$ converges
in $\tau$. So we can conclude that every $\tau$-open neighborhood
of $0$ in $\FR$ contains $A:=\{x^{k+1}\mid x\in D_{k+1}\}$. Fix
any $u,v\in A$ with $u\neq v$. Since $A\cap(u-v+A)\neq\emptyset$,
we know that the additive topology $\tau$ is not Hausdorff.
\end{proof}
Finally, we discuss the uniqueness of the pointwise convergence in
the Euclidean topology in the following sense:

Let $X$ be a set, and let $\tau_{1},\tau_{2}$ be two topologies
on $X$. We say that $\tau_{1}$ and $\tau_{2}$ are \emph{strongly
convergence equivalent} if a sequence $(x_{n})_{n\in\N}$ in $X$
converges to $x\in X$ in $\tau_{1}$ if and only if it converges
to $x$ in $\tau_{2}$. We have:
\begin{thm}
Let $X$ be a set, and let $\tau_{1},\tau_{2}$ be first-countable
topologies on $X$. Then $\tau_{1}$ and $\tau_{2}$ are strongly
convergence equivalent if and only if $\tau_{1}=\tau_{2}$.\end{thm}
\begin{proof}
It is enough to show that for any $\tau_{1}$-open subset $A$ of
$X$ and any $a\in A$, there exists a $\tau_{2}$-open neighborhood
$B$ of $a$ such that $B\subseteq A$. Assume that this is not true,
i.e., there exist a $\tau_{1}$-open subset $A$ of $X$ and $a\in A$
such that for any $\tau_{2}$-open neighborhood $B$ of $a$, $B\setminus A\neq\emptyset$.
Since $\tau_{2}$ is first-countable, there exists a countable $\tau_{2}$-neighborhood
basis $\{B_{i}\}_{i\in\N}$ of $a$ such that $\cdots\subseteq B_{1}\subseteq B_{0}$.
Pick $b_{i}\in B_{i}\setminus A$. Then the sequence $(b_{n})_{n\in\N}$
is a sequence which converges to $a$ in $\tau_{2}$. Since $\tau_{1}$
and $\tau_{2}$ are strongly convergence equivalent, this sequence
also converges to $a$ in $\tau_{1}$, which implies that $b_{i}\in A$
for $i$ large enough, i.e., we reach a contradiction.
\end{proof}
Note that all of the $\omega$-topology, the order topology and the
Euclidean topology are first-countable.

In particular of the above theorem, we have:
\begin{cor}
The pointwise convergence in the Euclidean topology is unique among
those in all first-countable topologies on Fermat reals.
\end{cor}

\subsection{Impossibility of Lebesgue dominated convergence}

As expected, the pointwise limit of a sequence of quasi-standard smooth
functions may not be quasi-standard smooth in any topology:
\begin{example}
\label{eg:ptwise-convergent} For $i=2,3,4,\ldots$, let $\sigma_{i}:\R\ra\R$
be an ordinary smooth function with the properties that $\sigma_{i}(\frac{1}{i})=1$,
$\im(\sigma_{i})=[0,1]_{\R}$, $\supp(\sigma_{i})\subseteq\big[\frac{1}{2}(\frac{1}{i+1}+\frac{1}{i}),\frac{1}{2}(\frac{1}{i-1}+\frac{1}{i})\big]_{\R}$,
and $\int_{0}^{1}\sigma_{i}(x)dx=1$%
\footnote{The last condition is a normalization, which will only be used in
the remark following this example.%
}. Then $(f_{n}:\FR\ra\FR)_{n\in\Z^{>0}}$ defined by 
\[
f_{n}(x)=\sum_{i=1}^{n}{^{\bullet}\sigma_{i+1}(x)}\cdot t^{1/i}
\]
is a sequence of quasi-standard smooth functions, which pointwise
converges to the function $f:\FR\ra\FR$ with 
\[
f(x)=\begin{cases}
{^{\bullet}\sigma_{i}(x)}\cdot t^{1/i}, & \textrm{if \ensuremath{x\in\big(\frac{1}{2}(\frac{1}{i+1}+\frac{1}{i}),\frac{1}{2}(\frac{1}{i-1}+\frac{1}{i})\big)}}\\
0, & \textrm{else}
\end{cases}
\]
in any topology (not necessarily additive Hausdorff). By Proposition~\ref{prop:qscoo},
$f$ is not quasi-standard smooth (around $0$).\end{example}
\begin{rem}
By the results we developed in the previous section, a similar statement
to the classical Lebesgue dominated convergence theorem does not hold
for quasi-standard smooth functions in either of the $\omega$-topology,
the order topology or the Euclidean topology. This is because, in
the above example, $|f_{n}(x)|\leq1$, but the limit of 
\[
\int_{0}^{1}f_{n}(x)dx=\sum_{i=1}^{n}t^{1/i}\cdot\int_{0}^{1}{^{\bullet}\sigma_{i+1}(x)}dx=\sum_{i=1}^{n}t^{1/i}\cdot\int_{0}^{1}\sigma_{i+1}(x)dx=\sum_{i=1}^{n}t^{1/i}
\]
as $n\ra\infty$ does not exist in the $\omega$-topology (Lemma~\ref{lem:ptwise-convergent-omega})
or the order topology (Lemma~\ref{lem:ptwise-convergent-order})
or the Euclidean topology (Definition~\ref{de:Euclidean}); for integration
of quasi-standard smooth functions, see~\cite{GW}.
\end{rem}
More generally, we have:
\begin{thm}
\label{thm:Lebesgue}The similar statement to the classical Lebesgue
dominated convergence theorem does not hold for $\FR$ with any topology
(Definition~\ref{de:topology}).\end{thm}
\begin{proof}
We prove below that if $\tau$ is an additive topology on $\FR$ such
that for every pointwise convergent sequence $(f_{n}:\FR\ra\FR)_{n\in\Z^{>0}}$
in $\tau$ with $|f_{n}|\leq g$ for some quasi-standard smooth function
$g:\FR\ra\FR$, we have the existence of $\lim_{n\ra\infty}\int_{0}^{1}f_{n}(x)dx$
in $\FR$ with respect to $\tau$, then $\tau$ cannot be Hausdorff.

In fact, we can take $g$ to be the constant function with value $1$.
Observe that in Example~\ref{eg:ptwise-convergent}, we can change
$t^{1/i}$ to any $\delta_{i}\in D_{\infty}$, so that we still have
a pointwise convergent sequence $(f_{n}^{(\delta_{i})_{i\in\Z^{>0}}})_{n\in\N}$
in any topology, and we have 
\begin{equation}
\int_{0}^{1}f_{n}^{(\delta_{i})_{i\in\Z^{>0}}}(x)dx=\sum_{i=1}^{n}\delta_{i}=:b_{n}.\label{eq:integral}
\end{equation}
By assumption, the sequence $(b_{n})_{n\in\Z^{>0}}$ converges in
$\tau$, i.e., for any $\tau$-open neighborhood $T$ of $0\in\FR$,
there exists $N\in\Z^{>0}$ such that for every $n,m>N$, $b_{n}-b_{m}\in T$.
In particular, $\delta_{l}\in T$ for $l$ large enough. So we get
the following information of $\tau$: every $\tau$-open neighborhood
of $0\in\FR$ contains all but finitely many points of $D_{\infty}$.
Therefore, $\tau$ cannot be Hausdorff.
\end{proof}

\section{\label{sec:IVP}Intermediate value property}

Recall that the Fermat reals $\FR$ has a total ordering, which is
an extension of the usual ordering on $\R$, and which makes $\FR$
an ordered commutative ring. In this section, we investigate which
quasi-standard smooth functions $f:\FU\ra\FR$, with $U$ an open
connected subset of $\R$, has the \emph{intermediate value property},
i.e., if $a,b\in\FU$ with $a<b$, then for any $y\in\FR$ between
$f(a)$ and $f(b)$, there exists some $c\in\FR$ between $a$ and
$b$ such that $f(c)=y$. 

Since quasi-standard smooth functions are locally restrictions (in
the sense of Fermat topology) of the Fermat extension of ordinary
smooth functions, we first discuss the question for Fermat extension
of ordinary smooth functions, which is answered in Corollary~\ref{cor:intermediate for Fermat extension}.
We subsequently use the idea of the proof of this corollary to get
a general criteria (Proposition~\ref{prop:intermediate}), together
with some applications. In order to reach these results, we will need
some preparations.

\subsection{The slice image theorem}

Here is the \emph{slice image theorem} for one variable, as one of
the key ingredients for solving the intermediate value property of
Fermat extension of ordinary smooth functions:
\begin{thm}
\label{thm:slice image theorem}Let $f:\R\ra\R$ be a smooth function,
and let $a\in\R$.
\begin{enumerate}[leftmargin=*,label=(\roman*),align=left ]
\item \label{enu:sliceFlatPoint}If $f^{(n)}(a)=0$ for all $n\in\Z^{>0}$,
then 
\[
\Ff(a+D_{\infty})=\{f(a)\};
\]

\item Assume that there exists some $n\in\Z^{>0}$ such that $f^{(n)}(a)\neq0$.
Denote by $m$ the smallest of such $n$.

\begin{enumerate}
\item \label{enu:slice_n_odd}If $m$ is odd, then
\[
\Ff(a+D_{\infty})=f(a)+D_{\infty};
\]

\item \label{enu:slice_n_even}If $m$ is even, then
\[
\Ff(a+D_{\infty})=f(a)+\sgn(f^{(m)}(a))\cdot D_{\infty}^{\geq0},
\]
where $D_{\infty}^{\ge0}:=\left\{ h\in D_{\infty}\mid h\ge0\right\} $.
\end{enumerate}
\end{enumerate}
\end{thm}
\begin{proof}
(1) This is clear from Taylor's expansion
\begin{equation}
\Ff(a+x)=f(a)+f'(a)x+\frac{f''(a)}{2!}x^{2}+\frac{f'''(a)}{3!}x^{3}+\cdots\label{eq:Taylor}
\end{equation}
for any $x\in D_{\infty}$. As usual, the sum in \eqref{eq:Taylor}
is finite since $x$ is nilpotent.

(2a) The inclusion $\Ff(a+D_{\infty})\subseteq f(a)+D_{\infty}$ follows
from \eqref{eq:Taylor}. Since $m$ is the smallest positive integer
such that $f^{(m)}(a)\neq0$, we have $m\geq1$, and 
\[
\Ff(a+x)=f(a)+\frac{f^{(m)}(a)}{m!}x^{m}+\frac{f^{(m+1)}(a)}{(m+1)!}x^{m+1}+\cdots
\]
for any $x\in D_{\infty}$. This is a finite sum since $x$ is nilpotent.
We introduce the following terminologies: given a quasi-decomposition
$z=\stz+\sum_{i=1}^{l}z_{i}t^{k_{i}}\in\FR$, we say that the \emph{leading
term} of $z-\stz$ is $z_{1}t^{k_{1}}$ with \emph{degree} $k_{1}$
and that the \emph{second leading term} is $z_{2}t^{k_{2}}$ with
\emph{degree} $k_{2}$. We will do a certain mixture of real and symbolic
calculations below. More precisely, first of all, we can think of
doing symbolic computations in the following algorithm using this
Taylor's expansion for all $f^{(i)}(a)$ with $i\geq m$, even if
some of them are already $0$, and during the computations, we will
not omit any terms computed via these terms in this Taylor's expansion,
even if the coefficient we get is already $0$; of course, by the
nilpotency of infinitesimals, only finitely many terms will stay.
Moreover, then from Step 1 on, we will fix a certain index $\eta$
at each step such that for the terms with degree $>\eta$ we do real
computations, i.e., we omit the terms with coefficient $=0$, and
for all the other terms, we keep everything computed symbolically.
Given any $w=f(a)+\sum_{i=1}^{k}w_{i}t^{a_{i}}\in f(a)+D_{\infty}$
as its \emph{decomposition}, we will find $c=a+\sum_{j=1}^{s}c_{j}t^{b_{j}}$
as its \emph{quasi-decomposition} from the following recursive algorithm
such that $\Ff(c)=w$. (Actually, for writing down the algorithm for
computing $c$, we can make the expression of $c$ as its decomposition
instead of quasi-decomposition. The mixture of real and symbolic calculations
and the quasi-decomposition of $c$ are very useful for estimating
the termination of the algorithm in finitely many steps at the end
of this proof.%
\footnote{After having a global picture of the idea of this proof, one could
instead think of a refined proof of using the decomposition of $c$
and real computations, and one will realize how complicated it becomes.%
}) For $F_{i}$ below, we mix real and symbolic computations (with
index $\eta_{i}$) as mentioned above, and for $G_{i}$, we do real
computations. 

Step 0: Let $G_{0}=w$ and let $F_{0}=f(a)$.

Step 1: Let $G_{1}=G_{0}-F_{0}=w-f(a)$, let $\eta_{1}=\frac{a_{1}}{m}$
and let $F_{1}=\Ff(a+c_{1}t^{b_{1}})-f(a)$, where $c_{1}$ and $b_{1}$
are chosen so that the leading terms of $F_{1}$ and $G_{1}$ are
equal. This is possible since we can set $b_{1}=\frac{a_{1}}{m}$,
and $c_{1}$ is the solution of the equation $f^{(m)}(a)x^{m}=m!w_{1}$
since $m$ is odd by assumption. Note that the second leading term
of $F_{1}$ has degree $(m+1)b_{1}$.

Step 2: Let $G_{2}=G_{1}-F_{1}=w-\Ff(a+c_{1}t^{b_{1}})$, let $\eta_{2}=\text{leading}(G_{2})-(m-1)b_{1}$
and let $F_{2}=\Ff(a+c_{1}t^{b_{1}}+c_{2}t^{b_{2}})-\Ff(a+c_{1}t^{b_{1}})$,
where $c_{1}$ and $b_{1}$ are determined in Step 1, and $c_{2}$
and $b_{2}$ are chosen so that the leading terms of $F_{2}$ and
$G_{2}$ are equal. This is possible since $b_{2}>b_{1}$ from the
requirement and Step 1%
\footnote{It is easy to check that if $b_{1}\geq b_{2}$, then the degree of
the leading term of $F_{2}$ is $mb_{2}$. So we have $b_{1}\geq b_{2}=\frac{\text{leading}(G_{2})}{m}>\frac{\text{leading }(G_{1})}{m}=\frac{a_{1}}{m}$,
contradicting that $b_{1}=\frac{a_{1}}{m}$ in the first step.%
}, and the leading term of $F_{2}$ is $\frac{f^{(m)}(a)}{(m-1)!}c_{1}^{m-1}c_{2}t^{(m-1)b_{1}+b_{2}}$.
Moreover, since the degree of the leading term of $G_{2}$ is either
$a_{2}$ or the degree of the second leading term of $F_{1}$, we
get 
\[
(m-1)b_{1}+b_{2}=\min\{a_{2},(m+1)b_{1}\}\leq(m+1)b_{1},
\]
i.e., $b_{2}\leq2b_{1}$. Hence, the second leading term of $F_{2}$
has degree $(m-2)b_{1}+2b_{2}$.

$\cdots$

Step r: Let $G_{r}=G_{r-1}-F_{r-1}=w-\Ff(a+\sum_{i=1}^{r-1}c_{i}t^{b_{i}})$,
let $\eta_{r}=\text{leading}(G_{r})-(m-1)b_{1}$ and let $F_{r}=\Ff(a+\sum_{i=1}^{r}c_{i}t^{b_{i}})-\Ff(a+\sum_{i=1}^{r-1}c_{i}t^{b_{i}})$,
where $(c_{1},b_{1}),\ldots,(c_{r-1},b_{r-1})$ are determined in
Step 1, $\cdots$, Step (r-1), and $c_{r}$ and $b_{r}$ are chosen
so that the leading terms of $F_{r}$ and $G_{r}$ are equal. This
is possible since $b_{r}>b_{r-1}$ from the requirement and Step (r-1)%
\footnote{It is easy to check that the degree of of the leading term of $F_{r}$
is $mb_{r}$ if $b_{1}\geq b_{r}$ or $(m-1)b_{1}+b_{r}$ otherwise.
The conclusion then follows easily.%
}, and the leading term of $F_{r}$ is $\frac{f^{(m)}(a)}{(m-1)!}c_{1}^{m-1}c_{r}t^{(m-1)b_{1}+b_{r}}$.
Note that the second leading term of $F_{r}$ has degree $(m-2)b_{1}+b_{2}+b_{r}$
since $b_{2}\leq2b_{1}$.

$\cdots$

Now we show that this procedure terminates in finitely many steps,
i.e., $G_{s}=o(t)$ for some $s\in\N$. Note that $G_{r}$ measures
the closeness of $a+\sum_{i=1}^{r-1}c_{i}t^{b_{i}}$ to the solution
of the equation $\Ff(x)=w$ at Step (r-1), and $F_{r}$ measures the
new extra terms created at Step r. From the above analysis, we know
that at Step r, there are only finitely many terms in $G_{r}$ which
has degree less than the degree of the second leading term of $F_{r}$,
and the new extra terms created in later on steps all have degree
greater than the degree of the second leading term of $F_{r}$%
\footnote{To see this, note that if exists, we can order the terms in $G_{r}$
by its degree which has degree strictly between the degree of the
leading term of $F_{r}$ (which is $(m-1)b_{1}+b_{r}$) and the degree
of the second leading term of $F_{r}$ (which is $(m-2)b_{1}+b_{2}+b_{r}$),
say they are $\{x_{1},x_{2},\cdots,x_{s}\}$ with $\deg(x_{1})<\deg(x_{2})<\cdots<\deg(x_{s})$.
Then $x_{1}$ becomes the leading term of $G_{r+1}$, and the second
leading term of $F_{r+1}$ is $(m-2)b_{1}+b_{2}+b_{r+1}$ which is
strictly greater than the second leading term of $F_{r}$. One can
then see that $x_{i}$ becomes the leading term of $G_{r+i}$, and
the second leading term of $F_{r+i}$ is $(m-2)b_{1}+b_{2}+b_{r+i}$
which is strictly greater than the second leading term of $F_{r+i}$,
for each $i=1,2,\cdots,s$. %
}. This is to say that for any fixed $r$, after finitely many steps,
the degree of the second leading term of $F_{r}$ becomes the degree
of the leading term of $G_{l}$ for some $l>r$. So for any $r\in\Z^{>0}$,
there exists $l>r$ such that $(m-1)b_{1}+b_{l}=(m-2)b_{1}+b_{2}+b_{r}$,
i.e., $b_{l}=b_{r}+(b_{2}-b_{1})$. This is to say that after finitely
many terms, the degree of the terms in $c$ will raise at least a
fixed positive constant, although we do not know explicitly how much
$b_{i}$ increases at each step. Therefore, the degree of some term
in $c$ will be greater than $1$ after finitely many steps, which
implies the termination of the algorithm after finitely many steps.

(2b) The proof is similar to (2a), except in Step 1 when solving the
equation $f^{(m)}(a)x^{m}=m!w_{1}$, since $m$ is even by assumption.
In this case, we can only solve this equation in $\R$ when $f^{(m)}(a)$
and $w_{1}$ have the same sign.\end{proof}
\begin{rem}
\label{rem:slice image}\ \end{rem}
\begin{enumerate}
\item As a warning, the proof of (2a) of the above theorem does not mean
that in that case the restriction map $\Ff|_{a+D_{\infty}}:a+D_{\infty}\ra\Ff(a)+D_{\infty}$
is injective. Instead, the algorithm in the proof gives the simplest
solution (called the \emph{fundamental solution}) to the equation
$\Ff(x)=w$ with real part $a$, in the sense that every solution
with real part $a$ is of this form possibly plus some more terms.
This can be proved using \eqref{enu:fundamental solution} below.
For example, by the algorithm, the equation $x^{3}=t$ has a solution
$x=t^{1/3}$. In fact, $x=t^{1/3}+y$ is a solution of this equation
as long as $y\in D_{2}$. In other words, the polynomial equation
$x^{3}=t$ has uncountably many solutions in $\FR$.
\item \label{enu:fundamental solution}Under the assumption of (2a), we
further assume that $y\in\FR$ is already a solution to the equation
$\Ff(x)=w$. Here is the procedure to get all the solutions of this
equation with real part $\sty$. First, we can refine $y=\sty+\sum_{i=1}^{n}\alpha_{i}t^{a_{i}}$
as its decomposition to get the fundamental solution: if $m$ is the
smallest positive integer such that $f^{(m)}(\sty)\neq0$, then $\tilde{y}=\sty+\sum_{i}\alpha_{i}t^{a_{i}}$
with the sum indexed by all $i=1,2,\ldots,n$ with $(m-1)a_{1}+a_{i}\leq1$,
is the fundamental solution. Then we get all the solutions: $\tilde{y}+z$
for $z\in D_{\infty}$ such that the degree of $z$ is greater than
\[
\begin{cases}
1-(m-1)a_{1,} & \text{if }ma_{1}\leq1\\
1/m, & \text{{otherwise.}}
\end{cases}
\]
We can get a similar result under the assumption of (2b), noting that
there are two fundamental solutions in that case when $w$ is not
real.
\item \label{enu:no intermediate value}Not every smooth function $\FR\ra\FR$
has the intermediate value property. For example, let $f:\R\ra\R$
be defined by
\[
f(x)=\begin{cases}
e^{-1/x}, & \text{if }x>0\\
0, & \text{else.}
\end{cases}
\]
Then $f$ is smooth, and $^{\bullet}(0,\infty)\cup\{0\}=\text{Im}(\Ff)$.
\item As a refinement of (2b) of the above theorem, we have
\[
\Ff(a+D_{\infty}^{\geq0})=f(a)+\sgn(f^{(m)}(a))\cdot D_{\infty}^{\geq0}=\Ff(a+D_{\infty}^{\leq0})
\]
under the same assumption.
\end{enumerate}
So we can determine the images of the Fermat extension of elementary
functions:
\begin{example}
\ \end{example}
\begin{enumerate}
\item Let $n\in\Z^{>0}$ and let $f:\R\ra\R$ be the function $x\mapsto x^{n}$.
Then 
\[
\text{Im}(\Ff)=\begin{cases}
\FR, & \text{if }n\text{ is odd}\\
\FR^{\geq0}, & \text{if }n\text{ is even.}
\end{cases}
\]

\item Let $f:\R\ra\R$ be $x\mapsto a^{x}$ with $0<a<1$ or $a>1$. Then
$^{\bullet}(0,\infty)=\text{Im}(\Ff)$.
\item Let $f:(0,\infty)\ra\R$ be $x\mapsto\log_{a}x$ with $0<a<1$ or
$a>1$. Then $\text{Im}(\Ff)=\FR$.
\item Let $f:\R\ra\R$ be either $x\mapsto\sin x$ or $x\mapsto\cos x$.
Then $\text{Im}(\Ff)=[-1,1]$.
\item Let $f:\R\ra\R$ be either $x\mapsto\tan x$ or $x\mapsto\cot x$.
Then $\text{Im}(\Ff)=\FR$.
\end{enumerate}

\subsection{\label{sub:Monotonicity}Monotonicity}

Here are some important observations, with the last one another important
ingredient for solving the intermediate value property for Fermat
extension of ordinary smooth functions:
\begin{enumerate}
\item \label{enu:observation1}There is no smooth function $f:\R\ra\R$
such that there exists a point $a\in\R$ with the properties that
for every solution $b$ of the equation $f(x)=a$, the smallest $m\in\Z^{>0}$
such that $f^{(m)}(b)\neq0$ exists and is even, say it is $m_{b}$,
and all these $f^{(m_{b})}(b)$'s are not of the same sign. \\
To prove this, one observes that $f(x)=a$ can only have finitely
many solutions on any closed interval by the evenness assumption.
Clearly it is impossible to connect the image of $f$ if $f(x)=a$
has two consecutive solutions $b$ and $c$, say $b<c$, (i.e., there
is no solution in the interval $(b,c)$), with $f^{(m_{b})}(b)\cdot f^{(m_{c})}(c)<0$
and $m_{b},m_{c}$ both even.
\item \label{enu:Boundary}(\emph{Boundary}) Let $f:\R\ra\R$ be a smooth
function. If $b\in\text{Im}(f)$ is minimum, then $\text{Im}(\Ff)\cap(b+D_{\infty})$
is either $\{b\}$ or $b+D_{\infty}^{\geq0}$. Dually, if $c\in\text{Im}(f)$
is maximum, then $\text{Im}(\Ff)\cap(c+D_{\infty})$ is either $\{c\}$
or $c+D_{\infty}^{\leq0}$. \\
These follow easily from Taylor's expansion of $f$. This leads to
the monotonicity discussed below.
\item \label{enu:global-monotonicity}(\emph{Global monotonicity}) Let $U$
be an open subset of $\R$, and let $f:U\ra\R$ be a smooth function.
If $f'(x)>0$ for all $x\in U$, then $\Ff:\FU\ra\FR$ is strictly
increasing, i.e., if $x,y\in\FU$ with $x<y$, then $\Ff(x)<\Ff(y)$.
More generally, if $f$ has the property that for any $u\in U$, there
exists some $m=m(u)\in\Z^{>0}$ such that $f^{(m)}(u)\neq0$, $m_{u}:=$
the smallest such $m$ is odd, and $f^{(m_{u})}(u)>0$, then $\Ff$
is increasing, i.e., for any $x,y\in\FU$ with $x<y$, we have $\Ff(x)\leq\Ff(y)$.
\\
(A typical such example is $f:\R\ra\R$ given by $f(x)=x^{3}$.)\\
Dually, if $f'(x)<0$ for all $x\in U$, then $\Ff:\FU\ra\FR$ is
strictly decreasing, i.e., if $x,y\in\FU$ with $x<y$, then $\Ff(x)>\Ff(y)$.
More generally, if $f$ has the property that for any $u\in U$, there
exists some $m=m(u)\in\Z^{>0}$ such that $f^{(m)}(u)\neq0$, $m_{u}:=$
the smallest such $m$ is odd and $f^{(m_{u})}(u)<0$, then $\Ff$
is decreasing, i.e., for any $x,y\in\FU$ with $x<y$, we have $\Ff(x)\geq\Ff(y)$.\\
Let us sketch the proof of the statements in the first paragraph,
since the others can be proved similarly. It is easy to show that
under these assumptions, the function $f$ is increasing, i.e., if
$u,u'\in U$ with $u<u'$, then $f(u)<f(u')$. So we only need to
prove the slice version of the statements, which will be in the next
observation.
\item \label{enu:slice-monotonicity}(\emph{Slice monotonicity}) Let $U$
be an open subset of $\R$, let $a\in U$ be a fixed point, and let
$f:U\ra\R$ be a smooth function. Assume that there exists $n\in\Z^{>0}$
such that $f^{(n)}(a)\neq0$. Let $m$ be the smallest such $n$.\\
Assume that $m$ is odd. If $f^{(m)}(a)>0$ (resp. $f^{(m)}(a)<0$),
then $\Ff|_{a+D_{\infty}}:a+D_{\infty}\ra f(a)+D_{\infty}$ is increasing
(resp. decreasing). \\
Assume that $m$ is even. If $f^{(m)}(a)>0$ (resp. $f^{(m)}(a)<0$),
then $\Ff|_{a+D_{\infty}^{\geq0}}:a+D_{\infty}^{\geq0}\ra f(a)+\sgn(f^{(m)}(a))\cdot D_{\infty}^{\geq0}$
is increasing (resp. decreasing), and $\Ff|_{a+D_{\infty}^{\leq0}}:a+D_{\infty}^{\leq0}\ra f(a)+\sgn(f^{(m)}(a))\cdot D_{\infty}^{\geq0}$
is decreasing (resp. increasing).\\
We prove the case when $m$ is odd and $f^{(m)}(a)>0$, since the
others can be proved similarly. For any $x_{1},x_{2}\in a+D_{\infty}$
with $x_{1}<x_{2}$, write $x_{1}=a+\sum_{i=1}^{n}\alpha_{i}t^{a_{i}}$
and $x_{2}=a+\sum_{i=1}^{n}\beta_{i}t^{a_{i}}$ for their quasi-decompositions
with at least one of $\alpha_{1}$ and $\beta_{1}$ non-zero, and
assume that $k\in\{1,2,\ldots,n\}$ is the smallest integer such that
$\alpha_{k}<\beta_{k}$. Let $x=a+\sum_{i=1}^{k-1}\alpha_{i}t^{a_{i}}$.
Note that Taylor's expansions of both $\Ff(x_{1})$ and $\Ff(x_{2})$
have $\Ff(x)$ in common. And the leading terms of $\Ff(x_{1})-\Ff(x)$
and $\Ff(x_{2})-\Ff(x)$ are 
\[
\begin{cases}
\frac{f^{(m)}(a)}{(m-1)!}\alpha_{1}^{m-1}\alpha_{k}t^{(m-1)a_{1}+a_{k}}\text{ and }\frac{f^{(m)}(a)}{(m-1)!}\alpha_{1}^{m-1}\beta_{k}t^{(m-1)a_{1}+a_{k}}, & \text{if }k>1\\
\frac{f^{(m)}(a)}{m!}\alpha_{k}^{m}t^{ma_{k}}\text{ and }\frac{f^{(m)}(a)}{m!}\beta_{k}^{m}t^{ma_{k}}, & \text{if }k=1,
\end{cases}
\]
respectively. Since $m$ is odd and $\alpha_{k}<\beta_{k}$, in both
cases we have $\Ff(x_{1})\leq\Ff(x_{2})$, or more precisely, $\Ff(x_{1})<\Ff(x_{2})$
if the degree of that leading term in the above expression is less
than or equal to $1$, and $\Ff(x_{1})=\Ff(x_{2})$ if the degree
is greater than $1$.
\end{enumerate}
Here is an application to transferring monotonicity:
\begin{prop}
Let $f:U\ra\R$ be a smooth map from an open subset of $\R$. Then
$f$ is increasing (resp. decreasing) if and only if $\Ff$ is increasing
(resp. decreasing).\end{prop}
\begin{proof}
The ``if'' part is clear. For the ``only if'' part, we are left
to deal with monotonicity within $f^{-1}(a)$ for every fixed $a\in\R$.
Note that if $f^{-1}(a)\neq\emptyset$, since $f$ is monotone, $f^{-1}(a)$
is the intersection of a closed interval (possibly a single point)
in $\R$ with $U$. Now we have two cases: (1) if the interior of
$f^{-1}(a)$ is non-empty, then every point in $f^{-1}(a)$ is a flat
point of $f$; (2) if $f^{-1}(a)$ contains a single point, denoted
by $b$, then either $b$ is a flat point of $f$, or there exists
some $m\in\Z^{>0}$ such that $f^{(m)}(b)\neq0$ and the smallest
such $m$ is odd. So one can apply Observation \eqref{enu:global-monotonicity},
the global monotonicity, to conclude the result.
\end{proof}
Here are some more complicated examples, which are complimentary to
Observation \eqref{enu:observation1} above:
\begin{example}
\ \label{exa:complicatedegs}\end{example}
\begin{enumerate}
\item Let $f:\R\ra\R$ be
\[
f(x)=\begin{cases}
x(x-1)^{2}e^{-1/x^{2}}, & \text{if }x\neq0\\
0, & \text{else}.
\end{cases}
\]
Then $f$ is smooth and surjective, $f(x)=0$ has two solutions $x=0$
and $x=1$, $f^{(n)}(0)=0$ for all $n\in\N$, the smallest $m\in\N$
such that $f^{(m)}(1)\neq0$ is $2$, and $f''(1)>0$. Therefore,
$0$ is in the interior of $\text{Im(f)}$, and $\text{Im}(\Ff)\cap D_{\infty}=D_{\infty}^{\geq0}$.
This example shows that if the ordinary smooth function has a flat
point, then it is possible that the image of its Fermat extension
has ``holes'', but not always (see the following one).
\item Let $f:\R\ra\R$ be
\[
f(x)=\begin{cases}
x(x-1)^{2}(x+1)^{2}e^{-1/x^{2}}, & \text{if }x\neq0\\
0, & \text{else}.
\end{cases}
\]
Then $f$ is smooth and surjective, $f(x)=0$ has three solutions
$x=0$, $x=1$ and $x=-1$, $f^{(n)}(0)=0$ for all $n\in\N$, the
smallest $m\in\N$ such that $f^{(m)}(\pm1)=0$ are both $2$, $f''(1)>0$
and $f''(-1)<0$. Therefore, $0$ is in the interior of $\text{Im}(f)$,
and $\text{Im}(\Ff)\cap D_{\infty}=D_{\infty}$.
\end{enumerate}

\subsection{Fermat extension of ordinary smooth functions}

Now we can prove the intermediate value property for certain Fermat
extension of ordinary smooth functions:
\begin{cor}
\label{cor:intermediate for Fermat extension}Let $U$ be an open
connected subset of $\R$, and let $f:U\ra\R$ be a smooth function
without any flat point, i.e., for any $u\in U$, there exists $m=m(u)\in\Z^{>0}$
such that $f^{(m)}(u)\neq0$. Then $\Ff$ has the intermediate value
property. \end{cor}
\begin{proof}
Let $a,b\in\FU$ with $a<b$. If $\Ff(a)=\Ff(b)$, then we are done.
Without loss of generality, we may assume that $\Ff(a)<\Ff(b)$. For
any $x\in\FR$ with $\Ff(a)<x<\Ff(b)$, we need to find $c\in\FR$
with $a<c<b$ (since $U$ is connected) such that $\Ff(c)=x$. We
prove this in several cases below.

Case 1: Assume $f(\sta)<\stx<f(\stb)$. Then by the classical intermediate
value theorem for $f$, the set $f^{-1}(\stx)\cap(\sta,\stb)_{\R}$
is non-empty. Now if there exists some $c_{0}$ in this set with the
property that the smallest positive integer $m$ such that $f^{(m)}(c_{0})\neq0$
is odd, or if there exist $c_{1},c_{2}$ in this set with the property
that the smallest positive integer $m_{i}$ such that $f^{(m_{i})}(c_{i})\neq0$
are both even and $f^{(m_{1})}(c_{1})f^{(m_{2})}(c_{2})<0$, then
we are done by Theorem~\ref{thm:slice image theorem}. So we may
assume that for every point $c_{0}$ in this set, the smallest positive
integer $m(c_{0})$ such that $f^{(m(c_{0}))}(c_{0})\neq0$ is even,
and all these $f^{(m(c_{0}))}(c_{0})$'s are of the same sign. Without
loss of generality, we may assume that they are all positive. This
implies that $\stx$ has to be the minimum of the smooth function
$f|_{(\sta,\stb)_{\R}}:(\sta,\stb)_{\R}\ra\R$, contradicting the
assumption that $f(\sta)<\stx<f(\stb)$ at the beginning of this case.
(Actually in this case, one can conclude together with Observation~\eqref{enu:observation1}
that there always exists $c_{0}\in f^{-1}(\stx)\cap(\sta,\stb)_{\R}$
with the property that the smallest positive integer $m$ such that
$f^{(m)}(c_{0})\neq0$ is odd.)

Case 2: Assume that $f(\sta)<f(\stb)$ and $\stx$ is equal to one
of them. Without loss of generality, we may assume $f(\sta)=\stx<f(\stb)$.
If $m$ is the smallest positive integer such that $f^{(m)}(\sta)\neq0$,
and $f^{(m)}(\sta)>0$, then we are done by Theorem~\ref{thm:slice image theorem}
together with slice monotonicity (Observation~\eqref{enu:slice-monotonicity}).
If $f^{(m)}(\sta)<0$, then $f$ is decreasing on $(\sta,\sta+\delta)_{\R}$
for some $\delta\in\R^{>0}$, and the claim then follows from Case
1.

Case 3: Assume that $f(\sta)=\stx=f(\stb)$. We may further assume
that $f^{(m)}(\sta)<0$, where $m$ is the smallest positive integer
such that $f^{(m)}(\sta)\neq0$, and $f^{(n)}(\stb)>0$ if $n$ is
even or $f^{(n)}(\stb)<0$ if $n$ is odd, where $n$ is the smallest
positive integer such that $f^{(n)}(\stb)\neq0$, since otherwise
the claim is true by slice monotonicity and Theorem~\ref{thm:slice image theorem}.
Under these assumptions, $f$ is decreasing near both $a$ and $b$,
and the claim then follows from Case 1.
\end{proof}
However, the converse of the above corollary is not true:
\begin{example}
\label{exa:flat-ivt}It is not true that if a smooth function $f:\R\ra\R$
has a flat point, then $\Ff$ does not have the intermediate value
property. For example, let $f$ be defined by 
\[
f(x)=\begin{cases}
e^{-1/x^{2}}\cos(1/x), & \text{if }x\neq0\\
0, & \text{else}.
\end{cases}
\]
Then one can check that (1) $f$ is a smooth function; (2) $x=0$
is the only flat point of $f$. The ``only'' part of the second
statement follows from the fact that the system of equations
\[
\begin{cases}
f'(x)=0\\
f''(x)=0
\end{cases}
\]
has only one solution: $x=0$. Note that $f^{-1}(0)=\{0,x_{k}\}_{k\in\Z}$
with $x_{k}=\frac{1}{k\pi+\frac{\pi}{2}}$, and $f'(x_{k})\neq0$
for all $k$. Together with Corollary~\ref{cor:intermediate for Fermat extension}
and the evenness of $f$, one can show that $f$ has the intermediate
value property.
\end{example}
Here is an application to extrema problems:
\begin{prop}
Let $f:\R\ra\R$ be a smooth map, and let $a,b\in\FR$ with $a<b$.
Then $\Ff|_{[a,b]}$ always has maximum and minimum. Moreover, if
$f$ has no flat point, then there exist $c,d\in\FR$ with $c<d$
such that $\Ff([a,b])=[c,d]$.\end{prop}
\begin{proof}
The first statement follows easily from the extreme value property
of $f|_{[\sta,\stb]_{\R}}$ together with Observation~\eqref{enu:Boundary}
(boundary) and Observation~\eqref{enu:slice-monotonicity} (slice
monotonicity), and the second statement then follows from Corollary~\ref{cor:intermediate for Fermat extension},
the intermediate value property for $\Ff$.
\end{proof}

\subsection{\label{sub:IV-for-quasi}Quasi-standard smooth functions}

Now we turn to the intermediate value property problem for a general
quasi-standard smooth function $g:\FU\ra\FR$, where $U$ is an open
connected subset of $\R$. For $a,b\in\FU$ with $a<b$, if there
exist $c_{1},\ldots,c_{n}\in\FU$ with $a<c_{1}<\ldots<c_{n}<b$ for
some $n\in\N$ such that $g$ has the intermediate value property
on each of $[a,c_{1}],[c_{1},c_{2}],\ldots,[c_{n},b]$, then it is
easy to see that $g$ also has the intermediate value property on
$[a,b]$. In other words, if each local expression of $g$ has the
intermediate value property, then so does $g$. So we are led to study
when the function $\Fh(v,\blank):\FU\ra\FR$ has the intermediate
value property, where $h:V\times U\ra\R$ is a smooth function with
$V$ an open subset of some Euclidean space, and $v\in\FV$ is a fixed
point. It is slightly more complicated than the Fermat extension $\Ff:\FU\ra\FR$
we have discussed in the previous subsection, because of the parameter
$v$. For example, let $g:\FR\ra\FR$ be defined by $g(y)=\Fh(t^{1/100},y)$,
where $h:\R\times\R\ra\R$ is a smooth function defined by $h(x,y)=y^{3}+xy^{2}$.
Then $g|_{D_{\infty}}:D_{\infty}\ra D_{\infty}$ is not surjective
since the equation $g(y)=-t^{51/100}$ has no solution. But $g(0)=0$
and $g(-1)=-1+t^{1/100}$, which implies that $g$ does not have the
intermediate value property over $\FR$. The main problem here is
$\deg_{y}(y^{3})>\deg_{y}(xy^{2})$, or in other words, it is not
determined whether $y^{3}$ or $xy^{2}$ will have the leading term
for variant $y$. After excluding such functions, we can prove the
following general criteria:
\begin{prop}
\label{prop:intermediate}Let $U$ be an open connected subset of
$\R$, and let $g:\FU\ra\FR$ be a quasi-standard smooth function
defined by $g(x)=\Fh(v,x)$, where $h:V\times U\ra\R$ is a smooth
function, $V$ is an open subset of some Euclidean space, and $v=\stv+\delta v\in\FV$
is a fixed point with $\delta v\neq0$. Assume that for each $x\in\FU$
there exists $m=m(x)\in\Z^{>0}$ such that $D^{(0,m)}h(\stv,\stx)\neq0$.
Write $m_{x}(=m_{\stx})$ for the smallest such $m$. If $D^{a}h(\stv,\stx)=0$
for any multi-index $a=(a_{v},a_{x})$ with $0<a_{x}<m_{x}$ for every
$x\in\FU$, then $g$ has the intermediate value property.
\end{prop}
Note that Corollary~\ref{cor:intermediate for Fermat extension}
applies to the case when $\delta v=0$.
\begin{proof}
The proof is very similar to that of Corollary~\ref{cor:intermediate for Fermat extension},
so we only sketch here.

The assumptions in the statement imply that if we omit all constant
terms in $g(x)=\Fh(v,x)$, i.e., if we consider $g(x)-g(\stx)$, then
the term $\frac{D^{(0,m_{x})}h(\stv,\stx)}{m_{x}!}x^{m_{x}}$ will
always provide the leading term when varying $x$.

Step 1: One can prove the slice image theorem in this case: Under
these assumptions,

(1) if $m_{x}$ is odd, then $g|_{\stx+D_{\infty}}:\stx+D_{\infty}\ra g(\stx)+D_{\infty}$
is surjective;

(2) if $m_{x}$ is even, then $g|_{\stx+D_{\infty}}:\stx+D_{\infty}\ra g(\stx)+\sgn(D^{(0,m_{x})}h(\stv,\stx))\cdot D_{\infty}^{\geq0}$
is surjective, or in the refined version, 
\[
g(\stx+D_{\infty}^{\geq0})=g(\stx)+\sgn(D^{(0,m_{x})}h(\stv,\stx))\cdot D_{\infty}^{\geq0}=g(\stx+D_{\infty}^{\leq0}).
\]

Step 2: One can prove the slice monotonicity in this case: Under these
assumptions,

(1) if $m_{x}$ is odd and $D^{(0,m_{x})}h(\stv,\stx)>0$ (resp. $D^{(0,m_{x})}h(\stv,\stx)<0$),
then $g|_{\stx+D_{\infty}}$ is increasing (resp. decreasing);

(2) if $m_{x}$ is even and $D^{(0,m_{x})}h(\stv,\stx)>0$ (resp.
$D^{(0,m_{x})}h(\stv,\stx)<0$), then $g|_{\stx+D_{\infty}^{\geq0}}$
is increasing (resp. decreasing) and $g|_{\stx+D_{\infty}^{\leq0}}$
is decreasing (resp. increasing).

Step 3: We can split into three cases depending on the order of the
real parts, and prove the intermediate value property as the proof
of Corollary~\ref{cor:intermediate for Fermat extension}.
\end{proof}
Here is a way to apply this proposition:
\begin{example}
Let $h:\R^{l}\times\R\ra\R$ be a smooth function with variables $(x,y)\in\R^{l}\times\R$,
and let $a\in\R^{l}$ be a fixed point. Assume that $h$ is a polynomial
in $y$ together with the property that there exists some $m\in\Z^{>0}$
such that $\frac{\partial^{m}h}{\partial y^{m}}(a,0)\neq0$. Then
there exist finitely many connected Fermat open subsets $A_{i}$ of
$\FR$ such that on each $A_{i}$ the smooth function $\Fh(v,\blank)$,
for any fixed $v\in a+D_{\infty}^{l}$, has the intermediate value
property. 

For example, let $h(x,y)=y^{3}+xy^{2}$ and let $v=t^{1/100}$. We
know from the paragraph above Proposition~\ref{prop:intermediate}
that $\Fh(v,\blank)$ does not have the intermediate value property
over $\FR$. Here is the procedure to find all connected Fermat open
subsets of $\FR$ on each of which $\Fh(v,\blank)$ has the intermediate
value property. Note that Taylor's expansion of $h$ at $(0,c)$ is
given by
\begin{align*}
h(v,c+y) & =(c+y)^{3}+(c+y)^{2}v\\
 & =c^{3}+c^{2}v+3c^{2}y+2cyv+3cy^{2}+y^{2}v+y^{3}.
\end{align*}
 According to Proposition~\ref{prop:intermediate}, $\Fh(v,\blank)$
has the intermediate value property on each connected Fermat open
subset as long as the subset does not contain $c$ with $3c^{2}=0$,
i.e., $c=0$. So we get such connected Fermat open subsets: $^{\bullet}(-\infty,0)$
and $^{\bullet}(0,\infty)$.\end{example}

\end{document}